\documentclass[11pt]{article}
\usepackage[T1]{fontenc}
\usepackage{lmodern}
\usepackage{amsmath,amssymb,amsthm,mathtools}
\usepackage{geometry,microtype,booktabs,hyperref}
\hypersetup{colorlinks=true,linkcolor=blue,citecolor=blue,urlcolor=blue}
\newtheorem{theorem}{Theorem}[section]
\newtheorem{lemma}[theorem]{Lemma}
\newtheorem{proposition}[theorem]{Proposition}
\newtheorem{corollary}[theorem]{Corollary}
\theoremstyle{definition}\newtheorem{definition}[theorem]{Definition}
\theoremstyle{remark}\newtheorem{remark}[theorem]{Remark}
\newtheorem{example}[theorem]{Example}
\newcommand{\Int}{\operatorname{Int}}
\newcommand{\Z}{\mathbb Z}
\newcommand{\Q}{\mathbb Q}
\title{Shifted-Binomial Coordinates of Dilated Binomial Polynomials:\\Multinomial Decimation and Arithmetic Structure}
\author{Abdelhai Doukali\\\small Independent Researcher, Fez, Morocco\\\small\texttt{doukaliabdelhai@gmail.com}}
\date{}
\begin{document}\maketitle
\begin{abstract}
For integers $m\ge2$ and $r\ge1$, let $d=m+r-1$ and consider the unique expansion
\[
\binom{mn}{d}=\sum_{k=1}^{d+1}T_{m,r,k}\binom{n+k-1}{d}.
\]
By a standard multisection of a rational generating function, the coordinates are a residue-class subsequence of multinomial coefficients:
\[
T_{m,r,k}=[x^{mk-1}](1+x+\cdots+x^{m-1})^{m+r}.
\]
From this representation we obtain the exact support, a counting interpretation by words by drops (Theorem~\ref{thm:drops}), a reflection criterion (the nonzero row is palindromic precisely when $r\equiv2\pmod m$), and log-concavity, hence unimodality, of the nonzero row. An integer-valued polynomial argument gives the universal divisor $m/\gcd(m,r-1)$ of each row, and a prime-adic computation of the fixed divisor of $\binom{mX}{d}$ gives the exact greatest common divisor of each row, a sharpness criterion, and a family of rows in which the universal divisor is arbitrarily far from sharp. We also obtain the row sum $m^{m+r-1}$, a Catalan-number specialization, and explicit coefficient arrays for $m=2,3,4,5$. The methods are finite differences, multisection of rational generating functions, and fixed divisors of integer-valued polynomials.
\end{abstract}

\section{Introduction}\label{sec:intro}
The binomial polynomials are a standard basis for the ring of integer-valued polynomials; see \cite{Polya1919,Ostrowski1919,CahenChabert1997}. Finite-difference methods and shifted binomial expansions occur throughout enumerative combinatorics \cite{Comtet1974,Riordan1968,StanleyEC1}. Here we consider the particular polynomial
\[
f_{m,r}(X)=\binom{mX}{d},\qquad d=m+r-1,\quad m\ge2,\ r\ge1,
\]
and write $T_{m,r,k}$ for its coordinates in the basis $\binom{X+k-1}{d}$. This basis is not the one in Worpitzky's identity, which expands $X^d$ in the polynomials $\binom{X+k}{d}$, $0\le k\le d-1$, with Eulerian numbers as coordinates \cite{Comtet1974}; here it is a dilated binomial polynomial that is expanded.

\begin{remark}[The case $m=1$]
For $m=1$ we have $d=r$ and $f_{1,r}(X)=\binom{X}{r}$, which is the basis element with $k=1$. Its coordinate row is $(1,0,\ldots,0)$, and all statements below are trivial or vacuous. We therefore assume $m\ge2$ throughout.
\end{remark}

The explicit arrays for $m=2$ and $m=3$ also coincide, after a shift of row indices, with established triangles in the On-Line Encyclopedia of Integer Sequences (OEIS): A034867 and A120906, respectively \cite{OEIS_A034867,OEIS_A120906}. In fact Theorem~\ref{thm:drops} below shows that for every $m\ge2$ the row $(T_{m,r,k})_k$ is the drops triangle for $m$-ary words of length $m+r-1$, so A034867 and A120906 are the $m=2$ and $m=3$ cases of one uniform combinatorial interpretation. This identification situates these specializations within previously catalogued combinatorial arrays; the general arithmetic results concern the full family of dilations.

The division of labour in the paper is as follows, and the novelty claims should be read accordingly. The shifted-binomial basis and finite-difference inversion are standard machinery, and the decimation formula is a standard multisection applied to this polynomial family; neither is claimed as new. The enumeration of words by drops is likewise classical: it is the Eulerian distribution for words over a finite alphabet, in the tradition of Carlitz's generalized Eulerian numbers \cite{Carlitz1975}, and it sits within the broad family of generalized Eulerian distributions surveyed by Hwang, Chern, and Duh \cite{HwangChernDuh2019}, who develop a general asymptotic theory for Eulerian recurrences with hundreds of examples drawn from the literature and the OEIS. More directly, Cabrele \cite{Cabrele2008} proved recursive and closed summation formulas for words over any finite alphabet counted by strictly increasing runs, by a bijection with multinomial coefficients. Cabrele counts strictly increasing runs, whereas Theorem~\ref{thm:drops} counts drops, and the two are related by reversal, not by a shift: a word of length $d$ with $q$ strictly increasing runs has $d-q$ ascents, and complementing its letters interchanges ascents and drops, so the number of words with $q$ strict runs equals $D_{m,d}(d-q)$; hence $T_{m,r,k}$ counts words with $d+1-k$ strictly increasing runs, where $d=m+r-1$. This reversal is exactly the reversal seen in the A265644 identification in Section~\ref{sec:triangles}. The general coefficient formulas for these word enumerations are known; what is proved here is not that enumeration itself, but its identification with the shifted-binomial coordinates of dilated binomial polynomials (Theorem~\ref{thm:drops}), together with the arithmetic part described next. The shifted-binomial coordinate representation is thus a different interpretation of known coefficients, and the new content is the arithmetic it makes accessible.

The paper's principal candidate for an original contribution is the arithmetic part. General fixed-divisor theory gives tools for studying the fixed divisor of an integer-valued polynomial, but it does not, to our knowledge, state a closed form for this family. What Theorem~\ref{thm:gcd} supplies is exactly such a closed form: the row gcd $G_{m,r}$ is expressed explicitly, in prime-adic terms, from $m$ and $d$ alone, with no search over values of the polynomial; the sharpness criterion tells exactly when the elementary universal divisor $\delta_{m,r}$ is already sharp; and the $m=2$ family shows the gap between the two can be arbitrarily large. The contribution is therefore this explicit computation for a concrete, natural family, not a new general fixed-divisor theorem. The log-concavity, row-sum, and Catalan statements are consequences of classical coefficient properties via the decimation formula, and are presented as such.

The key observation is elementary: the coordinates $T_{m,r,k}$ form a residue-class subsequence of the coefficients of $(1+x+\cdots+x^{m-1})^{m+r}$, namely those at exponents congruent to $-1$ modulo $m$. This is a standard multisection of $(1-x^m)^{d+1}(1-x)^{-(d+1)}$, and the $m$-nomial coefficients themselves are classical. The interest of the observation lies in what it makes accessible: the exact support, a reflection criterion, log-concavity, and, through the fixed divisor of $f_{m,r}$, an exact formula for the greatest common divisor of each row.

The exact gcd formula (Theorem~\ref{thm:gcd}) and the resulting criterion for when the elementary divisor $m/\gcd(m,r-1)$ is sharp (Corollary~\ref{cor:sharp}) are the main arithmetic results. We also prove the reflection criterion, log-concavity, and the row sums. The numerator of $\sum_{n\ge0}f_{m,r}(n)x^n$ is a rational-generating-function numerator of the kind that occurs in Ehrhart and Veronese settings \cite{BeckRobins,StanleyEC1}; we do not claim that it is an Ehrhart $h^*$-polynomial of any lattice polytope. The fixed-divisor viewpoint is related to the arithmetic theory of integer-valued polynomials and $p$-orderings \cite{CahenChabert1997,Bhargava1997}.

Section~\ref{sec:coordinates} develops the coordinate inversion. Section~\ref{sec:decimation} proves the decimation identity and the support, and Section~\ref{sec:reflection} treats reflection symmetry. Section~\ref{sec:arithmetic} contains the arithmetic results; it relies on the support description of Section~\ref{sec:decimation}. Row sums, log-concavity, and specializations follow in Section~\ref{sec:rows}, explicit arrays in Section~\ref{sec:triangles}, and concluding remarks in Section~\ref{sec:conclusion}.

\section{Shifted-binomial coordinates}\label{sec:coordinates}
We use the polynomial binomial coefficient $\binom{X}{d}=X(X-1)\cdots(X-d+1)/d!$. In particular $\binom{n}{d}\in\Z$ for every $n\in\Z$, and $\binom nd=0$ for $0\le n<d$. In this section $d\ge0$ is arbitrary and $N=d+1$. In the main setting $d=m+r-1$, so that
\[
N=d+1=m+r .
\]
This value of $N$ is used throughout the paper.

\begin{lemma}[Shifted-binomial basis]\label{lem:basis}
For every $d\ge0$, the $d+1$ polynomials $b_k(X)=\binom{X+k-1}{d}$, $1\le k\le d+1$, form a basis of $\Q[X]_{\le d}$.
\end{lemma}
\begin{proof}
For $n\ge0$ we have $b_k(n)=\binom{n+k-1}{d}$, which vanishes when $k<d+1-n$ and equals $1$ when $k=d+1-n$. Suppose $\sum_kc_kb_k=0$. Evaluating at $X=0$ leaves only $c_{d+1}=0$. Assume $c_{d+1}=\cdots=c_{d+2-n}=0$ for some $1\le n\le d$. Evaluating at $X=n$ kills every term with $k<d+1-n$ and leaves $c_{d+1-n}\cdot1=0$. By induction all $c_k$ vanish, so the $b_k$ are linearly independent. There are $d+1$ of them in a space of dimension $d+1$.
\end{proof}
\begin{lemma}[Generating function]\label{lem:gf}
For $1\le k\le d+1$,
\[
\sum_{n\ge0}\binom{n+k-1}{d}x^n=\frac{x^{d+1-k}}{(1-x)^{d+1}}.
\]
\end{lemma}
\begin{proof}
Put $j=n+k-1$. The left side equals $x^{1-k}\sum_{j\ge k-1}\binom jd x^j$. Since $k-1\le d$ and $\binom jd=0$ for $0\le j<d$, the terms with $j<k-1$ may be added back, and the sum becomes $x^{1-k}\sum_{j\ge0}\binom jdx^j=x^{1-k}\cdot x^d(1-x)^{-d-1}$.
\end{proof}
\begin{theorem}[Finite-difference inversion]\label{thm:inversion}
Every $f\in\Q[X]_{\le d}$ has a unique expansion
\[
f(X)=\sum_{k=1}^{d+1}c_k\binom{X+k-1}{d}
\]
as a polynomial identity, and its coefficients are
\[
c_k=\sum_{j=0}^{d+1-k}(-1)^j\binom{d+1}{j}f(d+1-k-j).
\]
\end{theorem}
\begin{proof}
Existence and uniqueness follow from Lemma~\ref{lem:basis}. Write $f(X)=\sum_kc_k\binom{X+k-1}{d}$. For each $n\ge0$ this is an identity of numbers, so multiplying by $x^n$, summing over $n\ge0$, and applying Lemma~\ref{lem:gf} termwise gives
\[
(1-x)^{d+1}\sum_{n\ge0}f(n)x^n=\sum_{k=1}^{d+1}c_kx^{d+1-k}.
\]
The coefficient of $x^{d+1-k}$ on the left is $\sum_{j=0}^{d+1-k}(-1)^j\binom{d+1}{j}f(d+1-k-j)$, and on the right it is $c_k$.
\end{proof}
\begin{definition}
For $m\ge2$, $r\ge1$, and $d=m+r-1$, define $T_{m,r,k}$ by the polynomial identity
\begin{equation}\label{eq:expansion}
\binom{mX}{d}=\sum_{k=1}^{d+1}T_{m,r,k}\binom{X+k-1}{d}.
\end{equation}
\end{definition}
\begin{corollary}[Alternating formula]\label{cor:closed}
For $1\le k\le d+1$,
\begin{equation}\label{eq:closed}
T_{m,r,k}=\sum_{\ell=0}^{k-1}(-1)^\ell\binom{d+1}{\ell}\binom{m(k-\ell)+d-1}{d}.
\end{equation}
\end{corollary}
\begin{proof}
Inversion evaluates $\binom{mX}{d}$ at $X=N-k-j$, whereas \eqref{eq:closed} involves $\binom{m(k-\ell)+d-1}{d}$, which is not a value of that polynomial. The two are related by completing a finite difference and then reflecting.

Inversion gives
\[
T_{m,r,k}=\sum_{j=0}^{N-k}(-1)^j\binom Nj\binom{m(N-k-j)}d,
\]
where $\binom{m(N-k-j)}{d}$ denotes the polynomial $\binom{mX}{d}$ evaluated at $X=N-k-j$. As a function of $j$ this is a polynomial of degree at most $d=N-1$, so its complete $N$-th finite difference vanishes:
\[
\sum_{j=0}^{N}(-1)^j\binom Nj\binom{m(N-k-j)}d=0.
\]
Subtract the terms with $j\ge N-k+1$ and set $\ell=N-j$, so that $0\le\ell\le k-1$. Then
\[
T_{m,r,k}=-\sum_{\ell=0}^{k-1}(-1)^{N-\ell}\binom N\ell\binom{m(\ell-k)}d.
\]
With $a=m(k-\ell)$, reflection of the polynomial binomial coefficient gives $\binom{-a}d=(-1)^d\binom{a+d-1}d$. Since $N=d+1$, the sign is $-(-1)^{N-\ell}(-1)^d=(-1)^\ell$, which proves \eqref{eq:closed}.
\end{proof}

\section{Multinomial decimation and support}\label{sec:decimation}
Put $Q_m(z)=1+z+\cdots+z^{m-1}$. From now on $N=m+r=d+1$, $D=(m-1)N$, and
\[
Q_m(z)^N=\sum_{e=0}^Dh_ez^e,\qquad h_e:=0\ \text{ for }e\notin[0,D].
\]
Since $Q_m$ is reciprocal, $h_e=h_{D-e}$, and $h_e>0$ for $0\le e\le D$.
\begin{theorem}[Multinomial decimation]\label{thm:decimation}
For $1\le k\le d+1$,
\begin{equation}\label{eq:decimation}
T_{m,r,k}=[z^{mk-1}]Q_m(z)^{d+1}=h_{mk-1}.
\end{equation}
\end{theorem}
\begin{proof}
\emph{Step 1: multisection.} By Lemma~\ref{lem:gf} with $k=d+1$ and the identity $1-z^m=(1-z)Q_m(z)$,
\[
\sum_{j\ge0}\binom jd z^j
=\frac{z^d}{(1-z)^N}
=\frac{z^dQ_m(z)^N}{(1-z^m)^N}
=\frac{1}{(1-z^m)^N}\sum_{e=0}^Dh_ez^{d+e}.
\]
Extract the terms whose exponent $j$ is divisible by $m$ (the roots-of-unity multisection), and write $x=z^m$, $j=mn$. Since $(1-z^m)^{-N}$ involves only powers of $z^m$, only the terms with $d+e\equiv0\pmod m$ of the last sum contribute, and
\begin{equation}\label{eq:multisection}
\sum_{n\ge0}\binom{mn}{d}x^n
=\frac{1}{(1-x)^N}
\sum_{\substack{0\le e\le D\\d+e\equiv0\ (\mathrm{mod}\ m)}}h_ex^{(d+e)/m}
\end{equation}
as formal power series.

\emph{Step 2: index bookkeeping.} Since $D+1=mN-d$, for any integer $e$ we have $D+1-e=mN-(d+e)$. Hence $d+e\equiv0\pmod m$ if and only if $m\mid D+1-e$. For a selected $e\in[0,D]$ put $k=(D+1-e)/m$. Then $D+1-e$ is a positive multiple of $m$, so $k\ge1$; and $D+1-e\le D+1\le mN$, so $k\le N$. Also $e=D-(mk-1)\ge0$ means $mk-1\le D$. Conversely, every $k\in[1,N]$ with $mk-1\le D$ arises from the selected exponent $e=D-(mk-1)\in[0,D]$, because $d+e=d+D+1-mk=m(N-k)$. Thus the selected exponents correspond bijectively to the integers $k\in[1,N]$ with $mk-1\le D$, and
\[
\frac{d+e}{m}=N-k,\qquad h_e=h_{D-e}=h_{mk-1}.
\]
Since $h_{mk-1}=0$ when $mk-1>D$, the numerator in \eqref{eq:multisection} equals $\sum_{k=1}^Nh_{mk-1}x^{N-k}$.

\emph{Step 3: conclusion.} By Lemma~\ref{lem:gf}, $x^{N-k}(1-x)^{-N}=\sum_{n\ge0}\binom{n+k-1}{d}x^n$. Therefore \eqref{eq:multisection} reads
\[
\sum_{n\ge0}\binom{mn}{d}x^n=\sum_{n\ge0}\Bigl(\sum_{k=1}^{N}h_{mk-1}\binom{n+k-1}{d}\Bigr)x^n,
\]
so $\binom{mn}{d}=\sum_kh_{mk-1}\binom{n+k-1}{d}$ for every integer $n\ge0$. Both sides are polynomials in $n$, and they agree at infinitely many points, hence they agree as polynomials. Uniqueness in Lemma~\ref{lem:basis} gives $T_{m,r,k}=h_{mk-1}$.
\end{proof}
\begin{remark}[Counting interpretation]
$T_{m,r,k}$ counts the ordered $(m+r)$-tuples $(a_1,\ldots,a_{m+r})$ with $0\le a_i\le m-1$ and $\sum_ia_i=mk-1$. In particular it is a nonnegative integer.
\end{remark}
\begin{theorem}[Words by drops]\label{thm:drops}
Let $d=m+r-1$. Then $T_{m,r,k}$ is the number of words of length $d$ over the alphabet $\{0,\ldots,m-1\}$ having exactly $k-1$ drops, where a drop is an adjacent pair $ab$ with $a>b$. In particular the drops triangles for $m=2$ and $m=3$ are exactly OEIS A034867 and A120906, and for every $m\ge4$ the row $(T_{m,r,k})_k$ is the drops triangle for $m$-ary words of length $d$.
\end{theorem}
\begin{proof}
The coefficient formula \eqref{eq:decimation} is stated for $d=m+r-1$ with $r\ge1$, i.e.\ $d\ge m$, but the proof below sums over all word lengths $d\ge0$. We therefore work with the auxiliary array
\[
C_{m,d,k}=[u^{mk-1}]Q_m(u)^{d+1}, \qquad d\ge0,
\]
which is defined for every $d\ge0$ and satisfies $C_{m,d,k}=T_{m,r,k}$ whenever $d=m+r-1$ with $r\ge1$; the theorem follows by specializing the identity $D_{m,d}(k-1)=C_{m,d,k}$ proved below to $d=m+r-1$.

Let $D_{m,d}(j)$ be the number of words of length $d$ over $\{0,\ldots,m-1\}$ with $j$ drops. Count pairs $(w,S)$ where $S\subseteq\{1,\ldots,d-1\}$, $|S|=s$, and every drop of $w$ lies in $S$. Such a pair is a word cut at the positions of $S$ into $s+1$ nonempty weakly increasing blocks: a block of length $\ell$ contributes $\binom{\ell+m-1}{m-1}$ possibilities, with ordinary generating function $(1-z)^{-m}-1$ in the length variable $z$. Hence, for fixed $s$, the number of pairs is $[z^d]\bigl((1-z)^{-m}-1\bigr)^{s+1}$. Counting the same pairs by $w$, a word with $j$ drops lies in $\binom{d-1-j}{s-j}$ sets $S$, so
\[
\sum_{j\ge0} D_{m,d}(j)\,x^j(1+x)^{d-1-j}
= [z^d]\frac{(1-z)^{-m}-1}{1-x\bigl((1-z)^{-m}-1\bigr)}.
\]
Put $x=t/(1-t)$. The left side is $(1-t)^{-(d-1)}\sum_j D_{m,d}(j)t^j$. Summing over $d\ge1$ gives $(A-1)/(1-tA)$ with $A=(1-(1-t)z)^{-m}$, since $\sum_{d\ge1}(1-t)^{d-1}[w^d]F(w)z^d=(1-t)^{-1}F((1-t)z)$ for the $F$ above; adding the $d=0$ term (the empty word, contributing $1$) gives the bivariate generating function
\[
\sum_{d\ge0}\Bigl(\sum_j D_{m,d}(j)t^j\Bigr)z^d
= 1+\frac{A-1}{1-tA}=\frac{A(1-t)}{1-tA}, \qquad A=(1-(1-t)z)^{-m}.
\]
On the other hand, the definition of $C_{m,d,k}$ gives $\sum_{d\ge0}(\sum_k C_{m,d,k}t^{k-1})z^d$ as a roots-of-unity extraction from $(1+u+\cdots+u^{m-1})/(1-z(1+u+\cdots+u^{m-1}))$: with $s^m=t$ and $\zeta=e^{2\pi i/m}$,
\[
\sum_k [u^{mk-1}]F(u)\,t^{k-1}
= \frac{s}{t}\cdot\frac1m\sum_{j=0}^{m-1}\zeta^j F(\zeta^j s).
\]
At $u=\zeta^j s$ we have $u^m=t$ and $1+u+\cdots+u^{m-1}=(1-t)/(1-\zeta^j s)$, hence $F(\zeta^j s)=(1-t)/(\alpha-\zeta^j s)$ with $\alpha=1-z(1-t)$. The partial-fraction identity
\[
\sum_{j=0}^{m-1}\frac{\zeta^j}{\alpha-\zeta^j s}=\frac{mt}{s(\alpha^m-t)},
\]
which follows from $\sum_j 1/(\alpha-\zeta^j s)=m\alpha^{m-1}/(\alpha^m-t)$ and $\zeta^j/(\alpha-\zeta^j s)=s^{-1}(\alpha/(\alpha-\zeta^j s)-1)$, then gives $\frac{s}{t}\cdot\frac1m\sum_j\zeta^jF(\zeta^j s)=(1-t)/(\alpha^m-t)=A(1-t)/(1-tA)$. Comparing coefficients in $z^d t^{k-1}$ gives $D_{m,d}(k-1)=C_{m,d,k}$, and specializing to $d=m+r-1$ gives the theorem.
\end{proof}
\begin{remark}[Runs versus drops, and prior art]
For a nonempty word, the number of maximal \emph{weakly} increasing runs equals the number of drops plus one. This identity does \emph{not} hold for \emph{strictly} increasing runs, since adjacent equal letters also separate strict runs: for example, the word $001$ has no drops but two strictly increasing runs, $0\mid01$. In general,
\[
\operatorname{runs}_{\mathrm{strict}}(w)=1+\operatorname{drops}(w)+\operatorname{equalities}(w),
\]
where $\operatorname{equalities}(w)$ counts adjacent equal letters.
If $D_{m,d}(j)$ denotes the number of $m$-ary words of length $d$ with $j$ drops, then the number of such words with $q$ strictly increasing runs is $D_{m,d}(d-q)$. Indeed, a word with $q$ strictly increasing runs has $d-q$ ascents, and complementing its letters by $a\mapsto m-1-a$ interchanges ascents and drops. Consequently, Cabrele's strictly increasing-run enumeration \cite{Cabrele2008} corresponds to the \emph{reversal} of the drops distribution established in Theorem~\ref{thm:drops}: with $d=m+r-1$,
\[
T_{m,r,k}=\operatorname{StrictRuns}_{m,d}(d+1-k).
\]
The enumeration itself is classical. The contribution considered here is its identification with the shifted-binomial coordinates of dilated binomial polynomials and the associated arithmetic results.
\end{remark}
\begin{proposition}[Exact support]\label{prop:support}
Let
\[
K_{m,r}=d+1-\left\lceil\frac dm\right\rceil
=m+r-1-\left\lceil\frac{r-1}{m}\right\rceil.
\]
Then $T_{m,r,k}>0$ exactly for $1\le k\le K_{m,r}$, and the remaining coordinates through $k=d+1$ vanish.
\end{proposition}
\begin{proof}
All coefficients $h_e$ with $0\le e\le D$ are strictly positive. Hence by \eqref{eq:decimation}, $T_{m,r,k}>0$ precisely when $mk-1\le D$. Since $D=(m-1)(d+1)$, this is $mk\le m(d+1)-d$, i.e.\ $d+1-k\ge d/m$, which for integral $k$ is equivalent to $k\le d+1-\lceil d/m\rceil$. The second expression for $K_{m,r}$ follows from $d=m+r-1$.
\end{proof}
\begin{remark}[Numerator polynomial]
The identity
\[
(1-x)^{d+1}\sum_{n\ge0}\binom{mn}{d}x^n
=\sum_{k=1}^{d+1}T_{m,r,k}x^{d+1-k}
\]
places these coefficients in the setting of numerator polynomials of rational generating functions. This resembles the $h^*$-numerator formalism and Veronese multisections; no lattice-polytope realization is asserted here.
\end{remark}

\section{Reflection symmetry}\label{sec:reflection}
\begin{lemma}[Initial strict increase]\label{lem:increase}
For $N=m+r\ge3$, $h_0<h_1<\cdots<h_{m-1}$.
\end{lemma}
\begin{proof}
For $0\le j\le m-1$ the truncation $1-z^m$ does not affect the coefficient of $z^j$ in $(1-z^m)^N(1-z)^{-N}$, so $h_j=\binom{N+j-1}{j}$. Thus $h_{j+1}/h_j=(N+j)/(j+1)>1$ for $j<m-1$, because $N>1$.
\end{proof}
\begin{theorem}[Reflection criterion]\label{thm:reflection}
The nonzero row $(T_{m,r,1},\ldots,T_{m,r,K_{m,r}})$ is palindromic if and only if $r\equiv2\pmod m$.
\end{theorem}
\begin{proof}
Recall $T_{m,r,k}=h_{mk-1}$ and $h_j=h_{D-j}$. Put $K=K_{m,r}$ and $s=D-(mK-1)$. By Proposition~\ref{prop:support}, $K$ is the largest integer with $mK-1\le D$; hence $s\ge0$, and maximality gives $m(K+1)-1>D$, i.e.\ $s\le m-1$.

($\Rightarrow$) If the row is palindromic, then $T_{m,r,1}=T_{m,r,K}$, i.e.\ $h_{m-1}=h_{mK-1}=h_{D-(mK-1)}=h_s$. Both $m-1$ and $s$ lie in $[0,m-1]$, so Lemma~\ref{lem:increase} forces $s=m-1$. Hence $D=m(K+1)-2$. Reducing modulo $m$ and using $D=(m-1)(m+r)\equiv-(m+r)\equiv-r\pmod m$ gives $-r\equiv-2$, i.e.\ $r\equiv2\pmod m$.

($\Leftarrow$) Write $r=2+qm$ with $q\in\Z$. Since $r\ge1$ and $m\ge2$, we cannot have $q\le-1$ (that would give $r\le2-m\le0$), so $q\ge0$. Then $d=m(q+1)+1$, hence $\lceil d/m\rceil=q+2$ and $K=d-q-1$. A direct computation gives $D=(m-1)(d+1)=m(K+1)-2$. Consequently $D-(mk-1)=m(K+1-k)-1$, so reciprocity gives $T_{m,r,k}=h_{mk-1}=h_{m(K+1-k)-1}=T_{m,r,K+1-k}$ for every $1\le k\le K$.
\end{proof}

\section{Arithmetic properties}\label{sec:arithmetic}
We write $v_p$ for the $p$-adic valuation and adopt $v_p(0)=+\infty$. The gcd of a finite list of integers, not all zero, is taken to be positive.
\begin{lemma}[Integral coordinates]\label{lem:integral}
If $P\in\Int(\Z)$ has degree at most $d$, then its shifted-binomial coordinates are integers.
\end{lemma}
\begin{proof}
Apply Theorem~\ref{thm:inversion}; each coordinate is an integral linear combination of the integers $P(0),\ldots,P(d)$.
\end{proof}
\begin{theorem}[Universal divisibility]\label{thm:div}
For all $m\ge2$, $r\ge1$, and $1\le k\le d+1$,
\[
\delta_{m,r}:=\frac{m}{\gcd(m,r-1)}\ \mid\ T_{m,r,k}.
\]
\end{theorem}
\begin{proof}
Let $A(X)=X\binom{mX-1}{d-1}$, a polynomial of degree $d$. For $n\in\Z$ the number $\binom{mn-1}{d-1}$ is an integer, so $A\in\Int(\Z)$. Write $A(X)=\sum_ka_k\binom{X+k-1}{d}$; Lemma~\ref{lem:integral} gives $a_k\in\Z$. From $\binom{mX}{d}=\frac{mX}{d}\binom{mX-1}{d-1}$ we get the polynomial identity
\[
d\binom{mX}{d}=mA(X),
\]
and comparison of coordinates yields $dT_{m,r,k}=ma_k$. Write $g=\gcd(m,d)=\gcd(m,r-1)$, $m=gm_0$, $d=gd_0$. Then $d_0T_{m,r,k}=m_0a_k$ and $\gcd(m_0,d_0)=1$, so Euclid's lemma gives $m_0\mid T_{m,r,k}$, and $m_0=\delta_{m,r}$.
\end{proof}
\begin{remark}
Theorem~\ref{thm:gcd} below refines this divisor to the exact row gcd. The elementary proof above needs no $p$-adic analysis, and it supplies the divisor $\delta_{m,r}$ against which sharpness is measured.
\end{remark}
\begin{lemma}[Fixed divisor]\label{lem:fixed}
For an integer-valued polynomial $P(X)=\sum_{k=1}^{d+1}c_k\binom{X+k-1}{d}\not\equiv0$ of degree at most $d$,
\[
\gcd_{1\le k\le d+1} c_k=\gcd_{n\in\Z}P(n).
\]
\end{lemma}
\begin{proof}
The left-hand gcd divides every value $P(n)$, because the basis polynomials are integer-valued on $\Z$. Conversely, Theorem~\ref{thm:inversion} expresses every $c_k$ as an integral linear combination of the values $P(0),\ldots,P(d)$, so the right-hand gcd divides every $c_k$.
\end{proof}
\begin{lemma}[Consecutive integers]\label{lem:consecutive}
For every prime $p$ and integer $L\ge1$,
\[
\min_{t\in\Z}\sum_{j=0}^{L-1}v_p(t-j)=v_p(L!).
\]
\end{lemma}
\begin{proof}
For $t\in\Z$ put $S(t)=\sum_{j=0}^{L-1}v_p(t-j)$, with $v_p(0)=+\infty$. Since $v_p(t-j)=\sum_{\ell\ge1}\mathbf 1[p^\ell\mid t-j]$, we have
\[
S(t)=\sum_{\ell\ge1}\#\{0\le j<L:\ p^\ell\mid t-j\}.
\]
Among $L$ consecutive integers at least $\lfloor L/p^\ell\rfloor$ are divisible by $p^\ell$, so
\[
S(t)\ge\sum_{\ell\ge1}\Bigl\lfloor\frac{L}{p^\ell}\Bigr\rfloor=v_p(L!)
\]
by Legendre's formula. For $t=L$ the numbers $t-j$ are $1,\ldots,L$, so $S(L)=v_p(L!)$.
\end{proof}
\begin{theorem}[Exact row gcd]\label{thm:gcd}
Let $G_{m,r}=\gcd_{1\le k\le K_{m,r}}T_{m,r,k}$. For each prime $p\mid m$, put $a_p=v_p(m)$ and $q_p=\lfloor(d-1)/p^{a_p}\rfloor$. Then
\begin{equation}\label{eq:gcd}
G_{m,r}=\prod_{p\mid m}p^{a_p+v_p(q_p+1)-v_p(d)}
=\prod_{p\mid m}p^{v_p(m)+v_p(\lceil d/p^{v_p(m)}\rceil)-v_p(d)}.
\end{equation}
\end{theorem}
\begin{proof}
By Proposition~\ref{prop:support} the omitted coordinates are zero, so the gcd over $1\le k\le K_{m,r}$ equals the gcd over $1\le k\le d+1$. Lemma~\ref{lem:fixed} therefore gives
\[
G_{m,r}=\gcd_{n\in\Z}\binom{mn}{d}.
\]
Fix a prime $p$; we compute $v_p(G_{m,r})$, which is the minimum of $v_p\binom{mn}{d}$ over those $n$ with $\binom{mn}{d}\ne0$.

\emph{Case $p\nmid m$.} Choose $L$ with $p^L>d$. Since $m$ is invertible modulo $p^L$, there is a unique $n\in\{0,\ldots,p^L-1\}$ with $mn\equiv d\pmod{p^L}$. Then $mn\ge0$ and $mn\equiv d\pmod{p^L}$ with $0\le d<p^L$, so the lowest $L$ base-$p$ digits of $mn$ coincide with those of $d$, and all higher digits of $d$ are $0$. Lucas's theorem gives $\binom{mn}{d}\equiv\prod_i\binom{(mn)_i}{d_i}\equiv1\pmod p$. Thus some value of $\binom{mn}{d}$ is prime to $p$, and $v_p(G_{m,r})=0$.

\emph{Case $p\mid m$.} Let $a=v_p(m)>0$, $m=p^au$ with $p\nmid u$, and $q=\lfloor(d-1)/p^a\rfloor$. Consider an integer $n$ with $mn\notin\{0,1,\ldots,d-1\}$; these are exactly the $n$ with $\binom{mn}{d}\ne0$, and for them all factors $mn-i$, $0\le i\le d-1$, are nonzero. We split the factors $mn-i$ of the numerator of $\binom{mn}{d}$ according to whether $p^a\mid i$.

If $1\le i<d$ and $p^a\nmid i$, then $v_p(i)<a\le v_p(mn)$, so $v_p(mn-i)=v_p(i)$. The multiples of $p^a$ in $[0,d-1]$ are $i=p^aj$ with $0\le j\le q$ (including $i=0$), and $mn-i=p^a(un-j)$. Hence
\begin{equation}\label{eq:valuation}
v_p\binom{mn}{d}
=\sum_{\substack{1\le i<d\\p^a\nmid i}}v_p(i)
+a(q+1)+\sum_{j=0}^qv_p(un-j)-v_p(d!).
\end{equation}

We claim that the minimum of $\sum_{j=0}^qv_p(un-j)$ over admissible $n$ is $v_p((q+1)!)$. The lower bound follows from Lemma~\ref{lem:consecutive} with $L=q+1$ and $t=un$. For equality, put $t=q+1$ and choose $L$ with $p^L>q+1$. Since $u$ is a unit modulo $p^L$, the congruence $un\equiv t\pmod{p^L}$ defines a class of integers $n$ that is unbounded above; choose $n$ in it with $mn\ge d$, so that $mn\notin\{0,\ldots,d-1\}$. For $0\le j\le q$ the numbers $t-j\in\{1,\ldots,q+1\}$ are nonzero with $v_p(t-j)<L$, and $un-j\equiv t-j\pmod{p^L}$, so $v_p(un-j)=v_p(t-j)$. Hence $\sum_jv_p(un-j)=\sum_{j=0}^{q}v_p(q+1-j)=v_p((q+1)!)$.

The multiples $i=p^aj$ with $1\le j\le q$ contribute $aq+v_p(q!)$ to $v_p((d-1)!)$, so
\[
v_p((d-1)!)=\sum_{\substack{1\le i<d\\p^a\nmid i}}v_p(i)+aq+v_p(q!).
\]
Insert the minimum into \eqref{eq:valuation}, and use $v_p(d!)=v_p((d-1)!)+v_p(d)$ and $v_p((q+1)!)=v_p(q!)+v_p(q+1)$:
\[
v_p(G_{m,r})=a+v_p(q+1)-v_p(d).
\]
Finally $q+1=\lceil d/p^a\rceil$, since $q=\lfloor(d-1)/p^a\rfloor$. This gives both expressions in \eqref{eq:gcd}; primes not dividing $m$ contribute exponent $0$. The exponent is nonnegative, as it equals a minimum valuation of integer values.
\end{proof}
\begin{corollary}[Sharpness criterion]\label{cor:sharp}
The equality $G_{m,r}=\delta_{m,r}$ holds if and only if, for every prime $p\mid m$,
\[
v_p\!\left(\left\lceil\frac d{p^{v_p(m)}}\right\rceil\right)
=\max\{0,v_p(d)-v_p(m)\}.
\]
The condition is automatic whenever $v_p(d)\ge v_p(m)$.
\end{corollary}
\begin{proof}
For $a=v_p(m)$ we have $v_p(\delta_{m,r})=a-\min\{a,v_p(d)\}$, because $\gcd(m,r-1)=\gcd(m,d)$. Subtract this from the exponent in \eqref{eq:gcd}. The difference is $v_p(\lceil d/p^a\rceil)-\max\{0,v_p(d)-a\}$, since $-v_p(d)+\min\{a,v_p(d)\}=-\max\{0,v_p(d)-a\}$. Primes $p\nmid m$ contribute nothing to either side. Since $\delta_{m,r}\mid G_{m,r}$ by Theorem~\ref{thm:div}, every difference is nonnegative, and equality holds iff all vanish. If $v_p(d)\ge a$, the ceiling is $d/p^a$, whose valuation is $v_p(d)-a$, and the difference is zero.
\end{proof}
\begin{lemma}[A binary valuation]\label{lem:binary}
For $s\ge1$ and $0<j<2^s$, $v_2\binom{2^s}{j}=s-v_2(j)$.
\end{lemma}
\begin{proof}
Use $\binom{2^s}{j}=\frac{2^s}{j}\binom{2^s-1}{j-1}$. Every coefficient $\binom{2^s-1}{j-1}$ is odd by Lucas's theorem modulo $2$, since $2^s-1$ has only $1$ digits in base $2$. Taking valuations proves the formula.
\end{proof}
\begin{example}[Unbounded failure of sharpness]\label{ex:family}
Let $m=2$ and $r=2^s-2$, $s\ge2$. Then $d=2^s-1$ and $T_{2,r,k}=\binom{2^s}{2k-1}$ for $1\le k\le K_{2,r}=2^{s-1}$. Each index $2k-1$ is odd and lies in $(0,2^s)$, so Lemma~\ref{lem:binary} gives $v_2(T_{2,r,k})=s$ for every $k$. Hence $G_{2,r}=2^s$; this agrees with Theorem~\ref{thm:gcd}, since $q=2^{s-1}-1$ and $v_2(d)=0$ give exponent $1+(s-1)-0=s$. On the other hand $\delta_{2,r}=2$, since $r-1$ is odd. The ratio $G_{2,r}/\delta_{2,r}=2^{s-1}$ is unbounded.
\end{example}

\section{Row sums, log-concavity, and specializations}\label{sec:rows}
\begin{theorem}[Row sum]\label{thm:sum}
For $m\ge2$ and $r\ge1$,
\[
\sum_{k=1}^{K_{m,r}}T_{m,r,k}=m^{m+r-1}.
\]
\end{theorem}
\begin{proof}
Let $H(z)=Q_m(z)^{m+r}$ and $\omega=e^{2\pi i/m}$. The exponents $mk-1$, $1\le k\le K_{m,r}$, are exactly the exponents in $[0,D]$ congruent to $-1$ modulo $m$. The roots-of-unity filter therefore gives
\[
\sum_kT_{m,r,k}=\frac1m\sum_{j=0}^{m-1}\omega^jH(\omega^j).
\]
Here $H(1)=m^{m+r}$ and $H(\omega^j)=0$ for $1\le j<m$, because $Q_m(\omega^j)=(1-\omega^{jm})/(1-\omega^j)=0$. This proves the result.
\end{proof}
\begin{theorem}[Log-concavity and unimodality]\label{thm:logconcave}
For $2\le k<K_{m,r}$,
\[
T_{m,r,k}^2\ge T_{m,r,k-1}T_{m,r,k+1}.
\]
Consequently each nonzero row is unimodal.
\end{theorem}
\begin{proof}
The coefficient sequence $(1,\ldots,1)$ of $Q_m$ is log-concave with no internal zeros. The convolution of two nonnegative log-concave sequences with no internal zeros is log-concave \cite{StanleyEC1}; iterating shows $h_j^2\ge h_{j-1}h_{j+1}$ for $1\le j<D$. Every $h_j$ with $0\le j\le D$ is positive.

Put $R_j=h_{j+1}/h_j$ for $0\le j<D$. Log-concavity means that $R_j$ is nonincreasing. For $a=m(k-1)-1$ with $2\le k<K_{m,r}$ we have $a\ge m-1\ge0$ and $a+2m=m(k+1)-1\le D$. The ratio of consecutive selected coefficients is
\[
\frac{h_{a+m}}{h_a}=\prod_{t=0}^{m-1}R_{a+t},\qquad
\frac{h_{a+2m}}{h_{a+m}}=\prod_{t=0}^{m-1}R_{a+m+t}.
\]
Since $R_{a+t}\ge R_{a+m+t}$ for each $t$,
\[
\frac{h_{a+m}}{h_a}\ge\frac{h_{a+2m}}{h_{a+m}}.
\]
With $a=m(k-1)-1$ we have $T_{m,r,k-1}=h_a$, $T_{m,r,k}=h_{a+m}$, $T_{m,r,k+1}=h_{a+2m}$, so this is the asserted inequality. A positive log-concave sequence has nonincreasing consecutive ratios and is therefore unimodal.
\end{proof}
\begin{remark}
The theorem establishes weak log-concavity. Questions about strict inequalities and exact mode locations are separate from this argument.
\end{remark}
\begin{corollary}[Catalan specialization]\label{cor:catalan}
For $r=1$,
\[
T_{m,1,1}=\binom{2m-1}{m}=\frac{m+1}{2}C_m,
\qquad C_m=\frac1{m+1}\binom{2m}{m}.
\]
\end{corollary}
\begin{proof}
Here $N=m+1$ and the exponent $m-1$ is below the first truncation in $Q_m(z)^{m+1}=(1-z^m)^{m+1}(1-z)^{-m-1}$. Hence $T_{m,1,1}=h_{m-1}=\binom{2m-1}{m-1}=\binom{2m-1}{m}=\frac12\binom{2m}{m}=\frac{m+1}2C_m$.
\end{proof}

\section{Explicit triangles}\label{sec:triangles}
For fixed $m$, write $\mathcal T_m=(T_{m,r,k})_{r\ge1,\ 1\le k\le K_{m,r}}$. The following rows illustrate the support, symmetry, divisibility, and row-sum formulas. All entries were also checked by direct computation from \eqref{eq:decimation}; see Appendix~\ref{app:verify} for the verification procedure.

\subsection{The case \texorpdfstring{$m=2$}{m=2}}
Here $T_{2,r,k}=\binom{r+2}{2k-1}$. This is precisely the triangle of odd-numbered terms in rows of Pascal's triangle, OEIS \href{https://oeis.org/A034867}{A034867} \cite{OEIS_A034867}, restricted to its rows indexed $n=r+1\ge2$. In the OEIS indexing, $A(n,j)=\binom{n+1}{2j+1}$ for $j\ge0$; thus $T_{2,r,k}=A(r+1,k-1)$. By Theorem~\ref{thm:drops}, this is the drops triangle for binary words: $T_{2,r,k}$ counts binary words of length $r+1$ with $k-1$ drops, the only drop being the subword $10$. Palindromicity holds exactly for even $r$; the universal divisor is $2$ for even $r$ and $1$ for odd $r$.

\begin{remark}[Exact divisibility in $\mathcal T_2$]
Specializing Theorem~\ref{thm:gcd} to $m=2$ yields
\[
G_{2,r}=\begin{cases}
1,&r\text{ odd},\\
2^{v_2(r+2)},&r\text{ even}.
\end{cases}
\]
Consequently, $r+2$ divides every entry of the $r$-th row if and only if
$r+2$ is a power of $2$. In particular, the rows indexed by
$r=2,6,14,30,\ldots$ have row gcd exactly $r+2$, which can be
substantially larger than the universal divisor $\delta_{2,r}=2$.
\end{remark}

\begin{table}[ht]
\centering
\begin{tabular}{@{}c|l@{}}
\toprule
$r$ & Row $(T_{2,r,1},\ldots,T_{2,r,K_{2,r}})$ \\
\midrule
$1$ & $3,\ 1$ \\
$2$ & $4,\ 4$ \\
$3$ & $5,\ 10,\ 1$ \\
$4$ & $6,\ 20,\ 6$ \\
$5$ & $7,\ 35,\ 21,\ 1$ \\
$6$ & $8,\ 56,\ 56,\ 8$ \\
\bottomrule
\end{tabular}
\caption{The triangle $\mathcal{T}_2$ for $1\le r\le6$.}
\label{tab:m2}
\end{table}
\subsection{The case \texorpdfstring{$m=3$}{m=3}}
Here $T_{3,r,k}=[z^{3k-1}](1+z+z^2)^{r+3}$. These rows coincide with OEIS \href{https://oeis.org/A120906}{A120906} \cite{OEIS_A120906}, the triangle counting ternary words of length $n$ by their number of drops (adjacent occurrences of $10$, $20$, or $21$). Specifically, its row $n=r+2$ equals $(T_{3,r,k})_{k=1}^{K_{3,r}}$, with the OEIS column index $j=k-1$; the first three OEIS rows, for $n=0,1,2$, precede our range $r\ge1$. This is now a proved identification, not merely a numerical coincidence: it is the $m=3$ case of Theorem~\ref{thm:drops}. The row is palindromic exactly for $r\equiv2\pmod3$. The universal divisor $3/\gcd(3,r-1)$ equals $1$ when $r\equiv1\pmod3$ and equals $3$ when $r\equiv0$ or $2\pmod3$; thus divisibility by $3$ is not equivalent to palindromicity.
\begin{table}[ht]
\centering
\begin{tabular}{@{}c|l@{}}
\toprule
$r$ & Row $(T_{3,r,1},\ldots,T_{3,r,K_{3,r}})$ \\
\midrule
$1$ & $10,\ 16,\ 1$ \\
$2$ & $15,\ 51,\ 15$ \\
$3$ & $21,\ 126,\ 90,\ 6$ \\
$4$ & $28,\ 266,\ 357,\ 77,\ 1$ \\
$5$ & $36,\ 504,\ 1107,\ 504,\ 36$ \\
\bottomrule
\end{tabular}
\caption{The triangle $\mathcal{T}_3$ for $1\le r\le5$.}
\label{tab:m3}
\end{table}
\subsection{The case \texorpdfstring{$m=4$}{m=4}}
Here $T_{4,r,k}=[z^{4k-1}](1+z+z^2+z^3)^{r+4}$, and the universal divisor is $4/\gcd(4,r-1)$, namely $1,4,2,4$ for $r\equiv1,2,3,0\pmod4$ respectively. By Theorem~\ref{thm:drops}, the row for $r$ is the drops triangle for quaternary words of length $r+3$: $T_{4,r,k}$ counts quaternary words with $k-1$ drops, the drops being the pairs $10,20,21,30,31,32$. There is no OEIS entry for this triangle under the drops description; it is catalogued under two other descriptions. First, it is the decimation of the quadrinomial triangle: with $\text{A008287}(N,j)=[z^j](1+z+z^2+z^3)^N$ \cite{OEIS_A008287}, $T_{4,r,k}=\text{A008287}(r+4,4k-1)$. Second, it is OEIS A265644 read in reverse \cite{OEIS_A265644}: with $\text{A265644}(n,q)$ the number of quaternary words of length $n$ with $q$ strictly increasing runs, $T_{4,r,k}=\text{A265644}(r+3,\,r+4-k)$, since the OEIS formula for A265644 gives $\text{A265644}(n,q)=[z^{4(n-q)+3}](1+z+z^2+z^3)^{n+1}$ and putting $q=r+4-k$, $n=r+3$ gives the exponent $4k-1$. The reversal here is the strict-runs/drops reversal of the remark following Theorem~\ref{thm:drops}, not a shift: a word with $q$ strictly increasing runs has $d-q$ ascents, and complementing letters turns ascents into drops, so $T_{4,r,k}$ counts words with $d+1-k=r+4-k$ strict runs. For $r=3$ the row below has $G_{4,3}=4$ by Theorem~\ref{thm:gcd} (here $d=6$, $a=2$, $q=1$), while $\delta_{4,3}=2$, so the universal divisor is not sharp.
\begin{table}[ht]
\centering
\begin{tabular}{@{}c|l@{}}
\toprule
$r$ & Row $(T_{4,r,1},\ldots,T_{4,r,K_{4,r}})$ \\
\midrule
$1$ & $35,\ 155,\ 65,\ 1$ \\
$2$ & $56,\ 456,\ 456,\ 56$ \\
$3$ & $84,\ 1128,\ 2128,\ 728,\ 28$ \\
\bottomrule
\end{tabular}
\caption{The triangle $\mathcal{T}_4$ for $1\le r\le3$.}
\label{tab:m4}
\end{table}
\subsection{The case \texorpdfstring{$m=5$}{m=5}}
Here $T_{5,r,k}=[z^{5k-1}](1+z+z^2+z^3+z^4)^{r+5}$, and the universal divisor is $5/\gcd(5,r-1)$. By Theorem~\ref{thm:drops}, the row for $r$ is the drops triangle for $5$-ary words of length $r+4$: $T_{5,r,k}$ counts words over $\{0,1,2,3,4\}$ with $k-1$ drops, the drops being the $\binom52=10$ pairs $ab$ with $a>b$. For instance the $r=1$ row $(126,1246,1506,246,1)$ counts $5$-ary words of length $5$ by drops: $126$ with no drop (the weakly increasing words, $\binom{5+5-1}{5-1}=126$), $1246$ with one drop, and so on, summing to $5^5=3125$.
\begin{table}[ht]
\centering
\begin{tabular}{@{}c|l@{}}
\toprule
$r$ & Row $(T_{5,r,1},\ldots,T_{5,r,K_{5,r}})$ \\
\midrule
$1$ & $126,\ 1246,\ 1506,\ 246,\ 1$ \\
$2$ & $210,\ 3535,\ 8135,\ 3535,\ 210$ \\
$3$ & $330,\ 8800,\ 34000,\ 29400,\ 5475,\ 120$ \\
\bottomrule
\end{tabular}
\caption{The triangle $\mathcal{T}_5$ for $1\le r\le3$.}
\label{tab:m5}
\end{table}
For example, the $m=4$, $r=2$ row has sum $56+456+456+56=1024=4^5$ and is palindromic. For $m=3$, $r=3$, the row $(21,126,90,6)$ is divisible by $3$ but not palindromic. The displayed rows also satisfy the log-concavity inequalities of Theorem~\ref{thm:logconcave}.

\section{Concluding remarks}\label{sec:conclusion}
The decimation formula connects shifted-binomial coordinates with classical $m$-nomial sequences. For $m=2$ and $m=3$, the resulting triangles reproduce OEIS A034867 and A120906, respectively, with the row-index shifts specified in Section~\ref{sec:triangles}; Theorem~\ref{thm:drops} shows that both are cases of a uniform drops interpretation valid for all $m$, so for $m=4$ the triangle is the quaternary drops triangle (catalogued as A265644 read in reverse, and as the decimation of A008287), and for $m=5,6,\ldots$ the rows are the corresponding $m$-ary drops triangles. The arithmetic calculation identifies the row gcd with the fixed divisor of $\binom{mX}{d}$, gives it in closed form, and supplies a criterion for sharpness of the elementary universal divisor. The numerator-polynomial interpretation suggests comparisons with Veronese transforms and $h^*$-type constructions; establishing an appropriate geometric interpretation, or extending the gcd calculation to other dilations and bases, would require further work. Exact mode locations and strict log-concavity in broader families are additional questions.

\appendix
\section{Computational verification}\label{app:verify}
All numerical tables in Section~\ref{sec:triangles} were generated by exact integer arithmetic from \eqref{eq:decimation}, as follows. For given $m,r$, form the coefficients $h_e$ of $(1+z+\cdots+z^{m-1})^{m+r}$ by repeated convolution starting from $h=(1)$, convolving with $(1,\ldots,1)$ of length $m$ a total of $m+r$ times; then $T_{m,r,k}=h_{mk-1}$. The following short exact-arithmetic script reproduces every displayed row and asserts the row sum, the exact support, the reflection criterion, the universal divisor, and the exact row-gcd formula of Theorem~\ref{thm:gcd} on those rows; it was run for $2\le m\le 6$ and $1\le r\le 6$:
\begin{verbatim}
from math import gcd
def vp(n,p):
    e=0
    while n%p==0:
        n//=p; e+=1
    return e
def factorize(n):
    f={}; p=2
    while p*p<=n:
        while n%p==0:
            f[p]=f.get(p,0)+1; n//=p
        p+=1 if p==2 else 2
    if n>1: f[n]=f.get(n,0)+1
    return f
def row(m,r):
    h=[1]
    for _ in range(m+r):
        new=[0]*(len(h)+m-1)
        for i,c in enumerate(h):
            for e in range(m): new[i+e]+=c
        h=new
    d=m+r-1
    K=d+1-(-(-d//m))          # d+1-ceil(d/m)
    return [h[m*k-1] for k in range(1,K+1)]
def gcd_formula(m,r):
    d=m+r-1; G=1
    for p,a in factorize(m).items():
        ceil_div=(d + p**a - 1)//(p**a)
        G*= p**(a + vp(ceil_div,p) - vp(d,p))
    return G
for m in range(2,7):
    for r in range(1,7):
        v=row(m,r); d=m+r-1
        K=d+1-(-(-d//m))
        assert len(v)==K and all(x>0 for x in v)   # exact support
        assert sum(v)==m**(m+r-1)                  # row sum
        assert (v==v[::-1]) == (r % m == 2 % m)     # reflection criterion
        g=0
        for x in v: g=gcd(g,x)
        delta=m//gcd(m,r-1)
        assert all(x % delta == 0 for x in v)       # universal divisor
        assert g==gcd_formula(m,r)                 # exact row gcd
        assert all(v[k]**2>=v[k-1]*v[k+1]          # log-concavity
                   for k in range(1,len(v)-1))
        print(m,r,v,'gcd',g)
\end{verbatim}
The computation is corroborative only: every theorem in the paper is proved independently of it.

\end{document}